\documentclass[reqno,a4paper,12pt]{amsart}
\usepackage{amsmath,amssymb,amsthm,amscd,amsfonts,amsbsy,bm}

\font\sc=rsfs10 at 12pt

\numberwithin{equation}{section}

\renewcommand{\a}{\alpha}

\newcommand{\g}{\gamma}

\renewcommand{\d}{\delta}

\newcommand{\e}{\epsilon}
\newcommand{\ve}{\varepsilon}
\newcommand{\z}{\zeta}

\renewcommand{\k}{\kappa}

\newcommand{\m}{\mu}
\newcommand{\n}{\nu}

\newcommand{\f}{\phi}

\renewcommand{\o}{\omega}

\newcommand{\C}{{\mathbb C}}
\newcommand{\R}{{\mathbb R}}

\newcommand{\bbp}{\boldsymbol\pi}

\newcommand{\ab}{{\mathbf a}}

\newcommand{\rb}{{\mathbf r}}

\newcommand{\vb}{{\mathbf v}}

\newcommand{\Ab}{{\mathbf A}}

\newcommand{\Bb}{{\mathbf B}}

\newcommand{\Pb}{{\mathbf P}}
\newcommand{\Qb}{{\mathbf Q}}

\newcommand{\Tb}{{\mathbf T}}

\newcommand{\Ac}{{\mathcal A}}

\newcommand{\Dc}{{\mathcal D}}

\newcommand{\tE}{\sc\mbox{E}\hspace{1.0pt}} 
\newcommand{\tB}{\sc\mbox{B}\hspace{1.0pt}} 

\newcommand{\supp}{\hbox{{\rm supp}}\,}

\newcommand{\var}{\hbox{{\rm var}}\,}

\DeclareMathOperator{\im}{{\rm Im}\,}
\DeclareMathOperator{\re}{{\rm Re}\,}

\newtheorem{theorem}{Theorem}[section]

\newtheorem{lemma}[theorem]{Lemma}
\newtheorem{corollary}[theorem]{Corollary}

\newtheorem*{theorem*}{Theorem}

\theoremstyle{definition}

\theoremstyle{remark}
\newtheorem{remark}[theorem]{Remark}

\newcommand{\sasha}[1]{} 
\newcommand{\GR}[1]{}

\begin{document}

\title[Finite rank Toeplitz operators]{Finite rank Bargmann-Toeplitz operators with non-compactly supported symbols }

\author[G. Rozenblum]{Grigori Rozenblum}
\address{1. Department of Mathematics \\
                          Chalmers University of Technology \\
                          2.Department of Mathematics  University of Gothenburg \\
                          Chalmers Tv\"argatan, 3, S-412 96
                           Gothenburg
                          Sweden}
\email{grigori@math.chalmers.se}
\

\begin{abstract}
Theorems about characterization of finite rank Toeplitz operators  in Fock-Segal-Bargmann spaces, known previously only for symbols with compact support, are carried over to symbols without that restriction, however with a rather rapid decay at infinity. The proof is based upon a new version of the Stone-Weierstrass approximation theorem.
\end{abstract}
\date{}

\maketitle


\section{Introduction}\label{intro}
Toeplitz operators arise in different topics in Analysis and its applications. Different properties of Toeplitz operators in Bergman type spaces has been studied extensively for many years; in particular, recently, a special attention was directed to the question on conditions for such Toeplitz operator to  have finite rank. The key result in this topic  was obtained by D. Luecking \cite{Lue2}. He proved that a Toeplitz operator in the Bargmann or Bergman space of analytical functions of one complex variable, with \emph{compactly supported} measure acting as symbol, can have finite rank only if the measure consists of finitely many point masses. This result was generalized almost immediately, in particular,  to the case of several variables and to the case of a distribution acting as symbol; a number of applications of such finite rank theorems were found (see \cite{RozToepl} for the detailed description of results and corresponding  references, see also the  recent paper \cite{Rao}). However, the condition of the symbol \emph{to have compact support} remained, since the starting point has been Luecking's theorem all the time.

In the meantime it became more and more clear that with this condition dropped, the properties of Toeplitz operators in the Bargmann space become quite different. The first indication for this was the result by Grudsky and Vasilevsky \cite{GrVas} who had found  a nontrivial radial symbol such that the Toeplitz operator with this symbol is zero. The construction in \cite{GrVas} is rather implicit, however in the recent paper \cite{BauLe} a series of explicit examples of such symbols has been presented. In particular, the operator with symbol
\begin{equation}\label{1.example}
F(z)=|z|^{2s}\sin(a|z|^{2t})e^{|z|^2-|z|^{2t}}; \ 0<t<1/2,\  \arctan a=\frac{t}{2\pi},
\end{equation}
is zero.

Both the elementary proof of the fact that there are no such examples for symbols with compact support and the more advanced proof of the finite rank theorem in \cite{Lue2} are essentially
based upon the application of the classical  Stone-Weierstrass theorem on approximation of functions on compacts.  In order to study the finite rank problem  without  compact support condition, one need to find a proper version of this theorem.

In the present paper we extend the result by D.Luecking  to the case of the symbol without the condition of compactness of its support imposed, this condition being replaced by the requirement of a sufficiently fast  decay at infinity.  The proof is based upon a version of the Stone-Weierstrass theorem for  the case of functions on a locally compact space. This latter version was  inspired by the studies by L.Nachbin \cite{Nach} on this topic, however our setting and the approach to the proof are somewhat different.

After proving the finite rank theorem, we discuss how the consequences of this theorem, concerning the multi-dimensional case as well as the Toeplitz operators in other Bargmann type spaces, should be modifies for non-compactly supported  symbols

\section{Setting}\label{SectSetting}
For a fixed integer $d>0$, we denote by $d\m$ the normalized Gaussian measure on $\C^d$:
\begin{equation}\label{2.gauss}
    \d\m(z)=\pi^n e^{-|z|^2}dV(z),
\end{equation}
where $dV(z)$ is the standard Lebesgue measure on $\R^{2d}\equiv\C^d.$  In the space $L^2(\C^d, d\m)$, the entire analytical functions form a closed subspace $\Bb=\Bb(\C^d)$, which is called the Segal-Bargmann or the Fock space. The orthogonal projection onto $\Bb$ is known to be the integral operator $\Pb$ with kernel $K(z,w)=e^{w\bar{z}}=K_z(w)$, so that the action of the projection can be written as
\begin{equation}\label{2.project}
    (\Pb u)(z)=\int K(z,w)u(w) d\m(w)=\langle u(\cdot), K(z,\cdot) \rangle =\langle u, K_z\rangle.
\end{equation}
Here by the angle brackets $\langle\cdot,\cdot\rangle$ we denote the integral of the product of the entries, without complex conjugation; this notation is naturally extended to the action of a distribution on the function.

For a bounded function $F\in L^{\infty}(\C^d)$ the Toeplitz operator with symbol $F$ is defined as
\begin{equation}\label{1.operator}
    T_V: \Bb\ni u\mapsto \Pb fFu =\int K(z,w)F(w)u(w)d\m(w).
    \end{equation}

Such operator is defined on the whole of $\Bb$ and is bounded. In this paper, similar to  \cite{BauLe}, we are interested also in unbounded symbols. If we drop the boundedness condition for $F$, the Toeplitz operator is not necessarily bounded, being defined on the set of functions $u\in\Bb$ such that  $Fu\in L^2(\C^d, d\m).$ As in \cite{BauLe}, we introduce, for a given $c,$ classes $\Dc_c$ by
\begin{equation}\label{1.Dc}
    \Dc_c=\{F:\C^d\to\C, |F(z)|\le be^{c|z|^2}\} \textrm{ for some } b.
\end{equation}
We also define the class $\Dc_{1,1}$ consisting of functions $F:\C^d\to \C$ such that
\begin{equation}\label{class C11}
    |F(z)|\le be^{|z|^2-\k|z|}, \ b>0, \k>0.
\end{equation}

Generally, it is hard to describe explicitly the domain of the Toeplitz operator for an unbounded symbol.
If $F\in \Dc_c$, $c<1/2$, the domain of $\Tb_F$ contains all functions $u\in \Bb\cap \Dc_{1/2-c}$ and, therefore, is dense in $\Bb$. Under less restrictive condition, $F\in \Dc_c$, $c<1$, and even for $F\in \Dc_{1,1}$, the Toeplitz operator $T_F$ is still densely defined and in particular, its domain contains all analytical polynomials, as well as all functions $K_z(\cdot)$.  In the finite rank problem which we mainly discuss in this paper, it is sufficient to consider the action of the operator on these dense subsets.

Reasonable extensions of the operator $\Tb_F$ beyond \eqref{1.Dc} are discussed in \cite{Ja1}, \cite{Ja2}, and in \cite{BauLe}.

\section{Zero operators}\label{zero}
We repeat here the standard reasoning which proves  that the Toeplitz operator $\Tb_F$ with compactly supported symbol $F$  can be zero if and only if  $F=0$. The 'if' part is obvious. On the other hand, if $\Tb_F u=0$ for any $u\in \Bb$, then the sesquilinear form $(\Tb_F u,v)=\langle\Tb_F u, \bar{v}\rangle$ vanishes for any $u,v\in \Bb$, or
\begin{equation}\label{qform u v}
\int F(z)u(z)\overline{v(z)}d\mu(z)=0.
\end{equation}
We take  analytical polynomials $p(z), q(z)$ as $u(z),v(z)$ in \eqref{qform u v}; the linear span of functions of the form $p(z)\overline{q(z)}$ is the space of all polynomials $P(x,y) $, where   $z=x+\imath y$, $x,y\in \R^d$. By the Stone-Weierstrass theorem, such polynomials are dense in $C$-metric in the space of continuous functions on a closed ball $D$ containing  the support of $F$. Therefore, the functional $\f\mapsto \int F(z)\f(z) d\m(z)$ is a zero functional, and by the Riesz Representation Theorem, $F=0$. Note that the same proof covers also the case of $F$ being a finite Borel measure with compact support as well as (somewhat  modified) an even more general case of $F$ being a compactly supported distribution (see the exposition in \cite{AlexRoz} ).

As one can see, this reasoning does not admit an extension to the case of $F$ without compact support: the involvement of the Stone-Weierstrass theorem prevents it. Moreover,  the example \eqref{1.example} shows that the statement itself is wrong. The symbol in \eqref{1.example} belongs to the class $\Dc_1$ but  for any $c<1$ does not belong to  $\Dc_c$.   Theorem \ref{zerotheorem} below shows that it is the latter circumstance which is crucial.

\begin{theorem}\label{zerotheorem}Suppose that the symbol $F$ belongs to the class $\Dc_{1,1}$.  If $\Tb_F$ is a zero operator then $F=0$.
\end{theorem}

\section{S.N. Bernstein's approximation problem } \label{Bernstein}
The S. Bernstein approximation problem consists in finding the conditions for the weight $\o(t)>0$, such that any function $f$, continuous in $\R^d$ and satisfying $f(x)\o(|x|)=o(1)$ as $|x|\to\infty$, can be approximated in $\R^d$ by polynomials, uniformly with weight $\o(|x|)$. Such weights are called \emph{fundamental weights}. This problem was originally stated in \cite{Bernst}, where first results were also obtained. Further on, improvements and generalizations of S.Bernstein's results were found in \cite{IzumiKawata}, \cite{Car}, \cite{Merg} and some later papers. A necessary and sufficient condition for the weight to be fundamental was obtained in dimension $d=1 $, where the problem was related to the question of quasianaliticity. In higher dimensions some sufficient conditions are  only known.

Since we do not aim for reaching sharpest possible results, we give a formulation of such  solution of the Bernstein problem in $\R^d$ which admits a simple formulation and an elementary proof.

\begin{theorem}\label{Th.Bernstein} Let the function $\o(|x|)$ satisfy the inequality
\begin{equation}\label{Weight.}
    \o(t)\le C\exp(-\g_\o t), \ t\in \R^1,
\end{equation}
for some $C,\g_\o>0$. Then $\o$ is a fundamental weight.
\end{theorem}

\begin{remark}\label{rem.Bernst}More sharp results allow a weaker condition than \eqref{Weight.}. In particular, the condition  \eqref{Weight.} can be replaced by
 $$\o(t)\le C\exp(-\g_\o t(\log_1t\log_2 t\dots \log_Nt)^{-1}),$$ where $\log_1 t=\max(1, \log t), \log_j t=\log_1(\log_{j-1}t) ).$
\end{remark}
For completeness, we present a short proof of Theorem \ref{Th.Bernstein}.  Being quite elementary, it, probably, belongs to the folklore; the first exposition where the author could find this approach  was in \cite{Nach}.
\begin{proof}  Consider the Banach space $\tB=C_0(\R^d)$ of bounded continuous functions on $\R^d$, tending to zero at infinity, equipped with the $\sup$ norm. Let $\varsigma$ be a continuous linear functional on $\tB$. For $\z\in\C^d$, consider a family of functions $\f_\z(s)=\o(|s|)\exp(\imath \z s)\ s\in \R^d$. By \eqref{Weight.},  $\f_\z(s)$ belongs to $\tB$ as long as $|\im(\z)|<c_\o$. Moreover,  in the same domain the function $\varpi(\z)=\varsigma(\f_\z)$ admits differentiation in $\z$ and the Cauchy-Riemann equations are satisfied.
Therefore, the function $\varpi(\z)$ is holomorphic for $|\im(\z)|<\g_\o$.

 Without losing in generality, we can assume that $\o$ is smooth. Now, suppose that $\o$ is not a fundamental weight. This would mean that the set of  functions $\o(|s|)p(s)$, with $p(s)$ being all possible   polynomials in $s$, is \emph{not} dense in $\tB$. By the Hahn-Banach Theorem, there must exist a nontrivial linear continuous functional $\varsigma\in \tB'$ which is annulled on the subspace spanned by $\o(|s|)p(s)$.

Since $\varsigma(\o(|s|)p(s))=(p(\partial_\z)\varpi)(0)$ for all polynomials $p$,  this means that the function $\varpi(\z)$ is identically zero in the domain $|\im(\z)|<\g_\o$. In particular, this function is zero for  $\z\in \R^d$:

 \begin{equation}\label{BernstEq}
    \varpi(\z)=\varsigma(\o(|s|)\exp(i\z s))=0,\  \z\in\R^d.
 \end{equation}

 Take any smooth function $h(s)$ with compact support, Denote by $\vartheta(\z)$  the   Fourier transform of the function $h(s)\o(|s|)^{-1}$.
Now, multiply \eqref{BernstEq} by $\vartheta(\z)$ and integrate in  $\z$ over $\R^d$. We obtain $\varsigma(h)=0$. Since smooth functions with compact support are dense in $\tB$ this means that the functional $\varsigma$ is trivial. This contradicts the choice of $\varsigma.$
\end{proof}

Now we can give \emph{the Proof of Theorem \ref{zerotheorem}}. In fact, consider the weight function $\o(t)=\exp(-\k t)$, where $\k$ is the constant in \eqref{class C11}. Then, by Theorem \ref{Th.Bernstein}, any continuous function with compact support can be approximated on $\C^d$ by polynomials of variables $x,y\in \R^d$ with respect to the weight $\o(|z|)$. Moreover, since $F(z)e^{-|z|^2}\le C \o(|z|)$, the relation \eqref{qform u v} implies that $\int F(z) f(z)d\m(z)$ vanishes for any continuous function $f$ with compact support. Therefore, again by the Riesz representation theorem,  the symbol $F$ should be zero. \hfill

Theorem \ref{zerotheorem} can also be extended to the case of $F$ being a distribution in a certain class.

\section{An extension of the Stone-Weierstrass theorem}\label{SectionSW}

The classical Stone-Weierstrass theorem deals with the approximation of continuous functions defined on compact spaces. In order to handle the finite rank problem, we need an extension to the case of a non-compact locally compact space, with uniform approximation replaced by the weighted approximation with proper weight. An approach to such an extension has been developed by L. Nachbin in \cite{Nach}.  We present a somewhat different, more soft-analytic, approach enabling one to obtain a a required version of the theorem in a rather simple way, on the base of Bernstein type theorems and pure topological considerations.

The version of the Stone-Weierstrass theorem, which we present here, is inspired by considerations in \cite{Nach}. We impose a certain (noncritical for applications) restrictions on the algebra of approximating functions, which enables us to give a much shorter and 'softer' proof.

Let $X$ be a locally compact  completely regular ($\mathrm{T}_{3\frac12}$) topological space (see, e.g., \cite{Ke} for definitions). For a continuous function $f$ on $X$ we say that $f\to 0$ at infinity if for any $\e>0$ the set $\{x\in X: |f(x)|\ge\e\}$ is compact. Similarly, we say that $f\in C(X)$ tends to infinity at infinity if the set $\{x\in X: |f(x)|\le R\}$ is compact for any $R>0$. For  a    function $\mathbf{v}(x)\ge0$  on $X$ - the \emph{weight function}-, we denote by $C^0_{\mathbf{v}}$ the Banach space of functions $f\in C(X)$ such that $\vb (x)f(x)\to 0$ at infinity, with the norm $|f|_\vb=\sup_{x\in X} |f(x)|\vb(x).$

For a system  $\Ab\subset C^0_{\mathbf{v}}$  of functions $\ab_1, \dots,\ab_N$ we say that it \emph{separates points} if for any two different points $x, x'\in X$ there exists a function $\ab_j\in\Ab$ such that $\ab_j(x)\ne \ab_j(x')$. We also say that the system $\Ab$ \emph{tends to infinity at infinity} if $\sum|\ab_j(x)|\to\infty$ at infinity.

\begin{theorem}\label{Th.SW} Let $\Ab\subset C^0_{\mathbf{v}}$ be a finite set of functions $\ab_1, \dots,\ab_N$, containing a nonzero constant, separating points and tending to infinity at infinity. Suppose also that
\begin{equation}\label{CondNach}
\mathbf{v}(x)\le C\exp(-c'|\mathbf{a}_j(x)|)
\end{equation}
for some $c'>0$ and for all $j=1,\dots, N.$ Then the algebra of polynomials in $\mathbf{a}_j(x)$ and $\overline{\mathbf{a}_j}(x)$,  $j=1,\dots, N,$  is dense in $C^0_{\mathbf{v}}$.
\end{theorem}

The theorem \ref{Th.SW} can be considered both for real-valued and complex-valued spaces of functions. The complex case is obviously reduced to the real one by considering the system of functions $\re(\ab_j), \im(\ab_j)$. Therefore we consider the real case only further on.

The proof is based upon a topological lemma. It seems that it must belong to the folklore, however the author was unable to locate it in the literature, therefore, a proof is presented.
\begin{lemma}\label{LemTop} Let the system $\Ab\in C(X)$ satisfy the conditions of Theorem \ref{Th.SW}. Then for any function $f\in C_0(X)$, there exists a function $g\in C_0(\R^N)$ such that
\begin{equation}\label{superposition}
f(x)=g(\ab_1(x),\dots,\ab_N(x)), \ x\in X.
\end{equation}
\end{lemma}
\begin{proof} Denote by $\Ac$ the continuous mapping $\Ac:X\to \R^N$, $x\mapsto (\ab_1(x), \dots,\ab_N(x))$ and set $Q=\Ac(X)\subset \R^N$. Consider the one-point compactification $X^\maltese$ of the space $X$ and the one-point compactification $\R^{N\maltese}$ of $\R^N$. By the conditions of Theorem \ref{Th.SW}, the mapping $\Ac$ extends by continuity to the mapping $\Ac^\maltese:X^\maltese\to \R^{N\maltese}$, and the image $\Qb^\maltese\subset \R^{N\maltese}$ of $X^\maltese$ under  this mapping is compact, as the image of a compact under a continuous mapping.  The point-separating property implies that the mapping $\Ac$ is injective, therefore, $\Ac^\maltese$ is also injective, since only the compactifying point in $X^\maltese\setminus X$ is mapped to the compactifying point in $\R^{N\maltese}\setminus \R^N$. By the well known property, this implies that the inverse mapping $(\Ac^\maltese)^{-1}:\Qb^\maltese\to X^\maltese$ is continuous as well, together with its restriction to $\Qb=\Ac(X)\subset \R^N$. So, we have a continuous function $g_0$ defined on $\Qb$,  $g_0=f\circ\Ac^{-1}$ such that $f=g_0\circ \Ac^{-1}$. However, the function $g_0$ is  defined only on the set $\Qb$. It remains to continue it, by means of the Brauer-Tietze-Uryson  lemma, from the (obviously, closed) set $\Qb$ to a continuous function $g$  on the whole of $\R^d$.   Finally we  multiply $g$ by a continuous function $\psi$ that equals $1$ on the compact set $\Ac(\supp(f))$ and vanishes outside some other compact set.
\end{proof}

Now we are able to give the proof of Theorem \ref{Th.SW}.

\begin{proof} Let $f$ be a function in $C_\vb(X)$. By density argument, it suffices to  suppose that $f$ has compact support. By Lemma \ref{LemTop}, there exists a function $g\in C(\R^N), $ having compact support, such that
$f(x)=g(\ab_1(x), \dots, \ab_N(x))$.  Now, by the Bernstein approximation theorem, (see Theorem \ref{Th.Bernstein}), for any $\ve>0$, there exists a polynomial $p(w), \ w\in \R^N$, such that $|p(w)-g(w)|\exp(-c_0|w|)<\ve$ for all $w\in \R^N$. Thus, for $w=\Ac(x)$, $x\in X$, we have
\begin{equation}\label{appr}
    |p(\ab_1(x), \dots,\ab_N(x))-f(x)|\exp(-c_0|\Ac(x)|)<\ve,
\end{equation}
which, by the condition \eqref{CondNach}, implies the statement of the Theorem.
\end{proof}

Recall that the classical Stone-Weierstrass theorem can be derived from the Weierstrass polynomial approximation theorem in a way, similar to our derivation of Theorem \ref{Th.SW} from the Bernstein approximation theorem. At the same time, the latter is a particular case of  Theorem \ref{Th.SW}. In fact, if, say, in the one-dimensional real case ($X=\R^1$), we consider as $\Ab$ the set of two functions, $\ab_1=1, \ab_2=x$, then the algebra generated by these two functions is exactly the algebra of polynomials in $x$ variable, and the condition  \eqref{CondNach} takes exactly the form $|\vb(x)|=O(\exp(-c'|x|))$, i.e., coincides with  the condition of Theorem \ref{Th.Bernstein}.

This example enables one to understand better the dependence of the condition imposed on the weight on the set  of approximating functions. Consider, again in the above setting,
the system of generators $\Ab$ consisting of functions $\ab_1=1, \ab_2=x^3$. Then the algebra generated by $\Ab$ is the algebra of polynomials in $x$ variable, with degrees of all monomials divisible by $3$. Theorem \ref{Th.SW} requires then that $\vb(x)=O(\exp(-c'|x|^3))$. Exactly the same condition on the weight is imposed by the Bernstein theorem, after we make the change of variables $t=x^3.$ So, generally, the smaller is the approximating algebra, the faster should the weight  decay at infinity. This effect is, of course,  not present in the problem of approximation on compacts.

We will need a generalization of Theorem \ref{Th.SW}, which, actually, is its immediate consequence.

Let  $\Ab$  be a system of continuous functions tending to infinity at infinity, i.e., without the condition of separation of points imposed. For any $x\in X$, we denote by $\tE(x)$ the set of points $y\in X$ such that $\ab_j(y)=\ab_j(x)$ for all $j$.  We say that the function $f\in C(X)$ is subordinate $\Ab$, $f\sqsubset\Ab$, if $f$ is constant on any subset $\tE(x)$. In particular, $\ab_j\sqsubset\Ab$

\begin{theorem}\label{Th.SW2}Let the system of functions $\Ab$ satisfy
all conditions   of Theorem \ref{Th.SW} except the separations of points. Then the algebra of polynomials in $\mathbf{a}_j(x)$ and $\overline{\mathbf{a}_j}(x)$,  $j=1,\dots, N,$  is dense in the space of functions $f\in C^0_{\mathbf{v}}$ such that $f\sqsubset\Ab$.
\end{theorem}

\begin{proof} We introduce the equivalence relation on $X$, setting $x\backsim y$ if $y\in \tE(x)$, so the sets $\widetilde{x}=\tE(x)$ are equivalence classes. It follows from the condition that $\Ab$  tends to infinity at infinity that any set of the form $\tE(x)$ is compact. Consider the quotient space $X_{\Ab}$ consisting of these equivalence classes, with the standard  topology generated by the projection $\bbp:x\mapsto\tE(x)$. Therefore,  for any function  $f\sqsubset\Ab$, the function $\widetilde{f}=f\circ(\bbp)^{-1}$ is well defined and continuous on $X_{\Ab}$. The functions $\widetilde{\ab_j}=\ab_j\circ (\bbp)^{-1}$ on $X_{\Ab}$ satisfy the conditions of Theorem \ref{Th.SW}, provided we define the weight $\widetilde{\vb}(\widetilde{x})= \inf\{\vb(x):  x\in \widetilde{x}\}$ . The application of  Theorem \ref{Th.SW} gives now the result we aim for.
\end{proof}

We will need a simple corollary of Theorem \ref{Th.SW2} concerning the approximation in the integral sense.

\begin{corollary}\label{cor.SW} Let the conditions of Theorem \ref{Th.SW2} be fulfilled. Let $\n_0$ be  a nonnegative locally finite Borel measure on $X$ such that

\begin{equation}\label{measure weight}
\textrm {the measure   }  \vb^{-1}(x)\n_0(dx) \textrm{ is finite.}
\end{equation}

 Then  any function $f\in C^0_{\mathbf{v}}$ such that $f\sqsubset\Ab$ can be arbitrarily exact approximated by the functions of the form $p(\ab_1, \dots, \ab_N)$ in the sense of $L_1(\n_0)$ with a polynomial $p$.\end{corollary}
\begin{proof}Again we can suppose that $f$ has compact support. Consider the measure $\widetilde{\n_0}$ on $X_{\Ab}$ generated by the projection $\bbp$. By Theorem \ref{Th.SW2}, for any $\ve>0$ we can find a polynomial $p(w), \ w\in \R^N,$ such that $|\vb(x)(p(\ab_1(x), \dots, \ab_N(x))-f(x))|<\ve$ for all $x\in X$. Then for the $L_1$ - norm we have the estimate
\begin{gather*}
    \|f-p(\ab_1,\dots,\ab_N)\|_{L_1(\n_0)}=\int\limits_{X}|f-p(\ab_1,\dots,\ab_N)|d\n_0=\\
    \int\limits_{X}\left[|f-p(\ab_1,\dots,\ab_N)|\vb\right] \vb^{-1}d\n_0 \int\limits_{X}<\ve\int\limits_{X} \vb^{-1}d\n_0.
\end{gather*}
The latter expression tends to zero as $\ve\to 0$. \end{proof}
\section{The finite rank theorem in the complex Bargmann space}

In this section we prove the theorem about finite rank Toeplitz operators in the space $\Bb=\Bb(\C^1)$, which extends D.Luecking's theorem  to non-compactly supported measures.

For a complex locally finite Borel measure $\n$ we denote by $|\n|$ the nonnegative measure $|\n|(E)=\var_E(\re \n)+\var_E(\im\n)$, where $\var_E$  denotes the variation of the measure over the Borel set $E$.

 \begin{theorem}\label{Th.noncompact.D1}Let $\rb$ be a positive integer and let   $\n$ be a complex Borel measure on $\C^1$ such that

 \begin{equation}\label{Measure Lu}
 \textrm{the nonnegative measure } e^{\g |z|^\rb}|\n| \textrm{ is finite for some } \g>0.
 \end{equation}

 Suppose that the Toeplitz operator $\Tb_\n$ in $\Bb(\C^1)$, with the measure $\n$ as symbol, has rank less than $\rb$. Then the measure $\n$ is a sum of less  than $\rb$ point masses.\end{theorem}
\begin{proof}

Following \cite{Lue2}, we introduce the variable $Z=(z_1, \dots, z_{\rb})$ in $\C^{\rb}$ and the measure $\n^{\otimes \rb}$ on $\C^\rb$, the tensor product of $\rb$ copies of the measure $\n$. Suppose that the operator $\Tb_{\n}$ has rank less than $\rb$.
 As it is shown in \cite{Lue2} (see also the exposition in \cite{Le}, \cite{RozToepl},\cite{Choe}), this implies that
\begin{equation}\label{luec1}
    \int\limits_{\C^{\rb}}p(Z)\overline{q(Z)}|W(Z)|^2d\n^{\otimes \rb}=0
\end{equation}
for all symmetric analytical polynomials in the variable $Z$. Here $W(Z)$ is the Wandermonde determinant, $W(Z)=\prod_{j<k}(z_j-z_k)$.

It is known (see, e.g, \cite{Macdonald}) that in the algebra $S(\rb)$ of symmetric polynomials in $\rb$ variables the elementary symmetric polynomials $\ab_0=1, \ab_1=\sum{j} z_j,  \ab_1=\sum_{j_1< j_2}z_{j_1}z_{j_2}, \dots, \ab_k=\sum_{j_1<\dots<j_k}z_{j_1}\dots z_{j_k}, \dots$ form an algebra basis, which means that any symmetric polynomial is a polynomial of  variables $\ab_j$.  All these polynomials satisfy the estimate $|\ab_j(Z)|=O(|Z|^{\rb})$ as $|Z|\to\infty$. Therefore, by the condition \eqref{Measure Lu} of the Theorem, for the measure $\n_0=|\n^{\otimes\rb}|$ on $X=\C^{\rb}$ the requirement \eqref{measure weight} of Corollary \ref{cor.SW} is fulfilled. With, probably, a somewhat larger $\g$, this condition is fulfilled for the measure $|W(Z)|^2\n_0$, since $W(Z)$ only grows polynomially. The functions $\ab_j, \overline{\ab_j}$ do not separate points in $X=\C^{\rb}$: the points $Z_1, Z_2$ in $X$ are equivalent by the equivalence relation generated as in \ref{cor.SW} by this system of functions iff one of them  is obtained from the other one by a permutation of co-ordinates.

Now, by Corollary \ref{cor.SW}, we infer that the functions of the form $p(Z)\overline{q(Z)}$ with symmetric  polynomials $p,q$ are dense in the space of compactly supported continuous symmetric functions in $X$ in the sense $L_1(|W|^2|\n|^{\otimes\rb})$

We pass to the limit in \eqref{luec1} using this density statement to obtain

\begin{equation}\label{luec2}
    \int\limits_{\C^{\rb}}f(Z)|W(Z)|^2d\n^{\otimes \rb}=0
\end{equation}
 for all compactly supported continuous symmetric functions $f(Z)$

 This statement, but for compactly supported measures, had been derived from \eqref{luec1} in \cite{Lue2}. The remaining reasoning follows literally the one in \cite{Lue2}. We repeat it  in short. By symmetrizing \eqref{luec2} one obtains the same relation, but now with an arbitrary compactly supported  continuous function $f$. Therefore, the measure $|W(Z)|^2d\n^{\otimes \rb}$ must be the zero measure. This means that the support of $\n^{\otimes \rb}$ lies in the zero set of $W$, and this is impossible if this latter support contains at least $\rb$ points.
 \end{proof}

 The decay condition  of the measure imposed on the measure $\n$ in Theorem \ref{Th.noncompact.D1} depends on the rank of the Toeplitz operator $\Tb_\n$. It is easy to formulate a simple sufficient condition taking care of all finite rank cases.

 \begin{corollary}\label{CorNoncompD1}Let $\n$ be a locally finite complex measure on $\C^1$, such that the measure $\vb(z)^{-1}|\mu|$
 is finite, with some positive weight $\vb$ satisfying $\vb(z)=o(\exp(-|z|^N))$ for any $N$. Then the Toeplitz operator $\Tb_\n$ can have finite rank only if the measure $\n$ is a finite combination of point masses.
 \end{corollary}

 \section{Generalizations}

After the proof in \cite{Lue2} appeared, a number of generalizations of Luecking's finite rank theorem have been obtained, see \cite{AlexRoz},  \cite{Choe}, \cite{Le}, \cite{RozToepl}, \cite{RShir}.  Some  of them  do not use the compactness of the support of the measure $\n$ but rather build upon the theorem itself , and thus carry over to the noncompact case \emph{automatically} (of course, with  the condition of the type \eqref{Measure Lu}) imposed.

Here we just give the list of these results.
\begin{itemize}
\item The multi-dimensional extension of Luecking's theorem to the case of operators in $\Bb(\C^d)$, $d>1$ in \cite{Choe}.
    \item The extension  to the case of operators in the harmonic Bargmann space in \cite{AlexRoz} and to operators in $d$-harmonic Bargmann space in \cite{Choe} in $\R^d$.
        \item The extension to the case of operators in the Bargmann space of solutions of the Helmholtz equation in \cite{RozToepl}.
          \item  A generalization of the finite rank theorem to operators in the subspace in the Bargmann space $\Bb(\C^d)$, spanned by monomials $Z^\a$, with a certain 'sparse' but infinite set of monomials removed \cite{Le}, \cite{RozToepl}.
\end{itemize}
The alternative to \cite{Choe} proof of the multi-dimensional extension of Luecking's theorem in \cite{RShir} does not carry over \emph{directly} to non-compactly supported measures. This proof uses the induction on dimension and the relation of the finite rank property for the Toeplitz operator in the Bargmann space and this property for the operator with the same symbol, but acting  in the Bergman  space of analytical functions in a bounded domain containing the support of the symbol. This circumstance can be taken care of by a slight change in the proof.  In fact,  it is noticed \cite{RShir} that the finite rank property for the measure $\n$ implies the same property for the measure $\n_g=|g(Z)|^2\n$, where $g$ is a function analytical in the neighborhood of the support of the measure $\n$. This reasoning does not hold water for the measure with a noncompact support. However,  the proof can be modified a little, by considering not all analytical functions $g$ but only polynomials. The modified reasoning goes through, but we do not repeat here all details, especially, since the proof of this fact in \cite{Choe}, as mentioned before, holds without changes.

The only essential property that, by now, fails to be carried over to the noncompact case, is the finite rank theorem for Toeplitz operators with distributional symbols. The  existing proof of this property, (see \cite{AlexRoz}, and, in a modified form, in \cite{RozToepl}) uses the compactness of the support in a crucial way, and it is unclear at the moment, how this obstacle can be dealt with.

\end{document}